\documentclass[12pt,twoside,reqno,psamsfonts]{amsart}

\usepackage[OT1]{fontenc}
\usepackage{type1cm}
\usepackage{amssymb}
\usepackage[dvips]{graphicx}
\usepackage{psfrag}          
\usepackage{geometry}        
\usepackage{version}         

\numberwithin{equation}{section}

\theoremstyle{plain}
\newtheorem{theorem}{Theorem}[section]
\newtheorem{corollary}[theorem]{Corollary}
\newtheorem{proposition}[theorem]{Proposition}
\newtheorem{lemma}[theorem]{Lemma}

\theoremstyle{remark}
\newtheorem{remark}[theorem]{Remark}

\theoremstyle{definition}

\newcommand{\HH}{\mathcal{H}}
\newcommand{\II}{\mathcal{I}}

\newcommand{\R}{\mathbb{R}}

\newcommand{\N}{\mathbb{N}}
\newcommand{\hhh}{\mathtt{h}}
\newcommand{\iii}{\mathtt{i}}
\newcommand{\jjj}{\mathtt{j}}
\newcommand{\kkk}{\mathtt{k}}
\newcommand{\eps}{\varepsilon}

\newcommand{\roo}{\varrho}

\newcommand{\yli}[2]{\genfrac{}{}{0pt}{}{#1}{#2}}

\DeclareMathOperator{\dimh}{dim_H}
\DeclareMathOperator{\dimp}{dim_p}
\DeclareMathOperator{\dist}{dist}
\DeclareMathOperator{\diam}{diam}
\DeclareMathOperator{\proj}{proj}

\DeclareMathOperator{\spt}{spt}

\DeclareMathOperator{\udimloc}{\overline{dim}_{loc}}
\DeclareMathOperator{\ldimloc}{\underline{dim}_{loc}}

\begin{document}

\title{Local conical dimensions for measures}

\author{De-Jun Feng}
\address{Department of Mathematics \\
         The Chinese University of Hong Kong \\
         Shatin \\
         Hong Kong}
\email{djfeng@math.cuhk.edu.hk}

\author{Antti K\"aenm\"aki}
\address{Department of Mathematics and Statistics \\
         P.O.\ Box 35 (MaD) \\
         FI-40014 University of Jyv\"askyl\"a \\
         Finland}
\email{antti.kaenmaki@jyu.fi}

\author{Ville Suomala}
\address{Department of Mathematical Sciences\\
         P.O.\ Box 3000\\
         FI-90014 University of Oulu\\
         Finland}
\email{ville.suomala@oulu.fi}

\date{\today}

\begin{abstract}
We study the decay of $\mu(B(x,r)\cap C)/\mu(B(x,r))$ as $r\downarrow 0$ for different kinds of measures $\mu$ on $\R^n$ and various cones $C$ around $x$. As an application, we provide sufficient conditions implying that the local dimensions can be calculated via cones almost everywhere.
\end{abstract}

\maketitle

\section{Introduction and notation}

Let $\mu$ be a measure on $\R^n$ and let $C(x) \subset \R^n$ be a cone with a vertex at $x \in \R^n$.
Our motivation for this article stems from the following question: For what types of cones and under what assumptions on the measure we have
\begin{align}
  \label{eq:uconicaldim}
  \limsup_{r \downarrow 0} \frac{\log \mu\bigl( B(x,r) \cap C(x) \bigr)}{\log r} &= \limsup_{r \downarrow 0} \frac{\log \mu\bigl( B(x,r) \bigr)}{\log r}, \\
  \label{eq:lconicaldim}
  \liminf_{r \downarrow 0} \frac{\log \mu\bigl( B(x,r) \cap C(x) \bigr)}{\log r} &= \liminf_{r \downarrow 0} \frac{\log \mu\bigl( B(x,r) \bigr)}{\log r},
\end{align}
for $\mu$-almost all points $x \in \R^n$? Here the right-hand sides of \eqref{eq:uconicaldim} and \eqref{eq:lconicaldim} are denoted by $\udimloc(\mu,x)$ and $\ldimloc(\mu,x)$, and they are the \emph{upper and lower local dimensions of the measure $\mu$ at $x \in \R^d$}, respectively. 

We prove that if $C$ is a cone with opening angle at least $\pi$, then  \eqref{eq:uconicaldim} and \eqref{eq:lconicaldim} hold for all measures $\mu$ and for $\mu$-almost all $x\in\R^n$. Moreover, \eqref{eq:lconicaldim} holds also for cones with small opening angle at $\mu$-almost all points where
$\ldimloc(\mu,x)$ is large. The analogous result for $\udimloc(\mu,x)$ fails. Finally, we prove that  \eqref{eq:uconicaldim} and  \eqref{eq:lconicaldim} are true for any purely unrectifiable self-similar measure on a self-similar set satisfying the open set condition. Most of the results are obtained as corollaries to more general theorems describing the speed of decay of
\begin{equation}\label{cdens}
  \frac{\mu\bigl( B(x,r) \cap C(x) \bigr)}{\mu\bigl( B(x,r) \bigr)}
\end{equation}
as $r\downarrow 0$.
In geometric measure theory, it has been of great interest to determine when the (upper and lower) limits of \eqref{cdens} are zero (resp. positive) at $\mu$-almost all points. The results obtained so far have connections and applications to rectifiability and porosity problems, see e.g. \cite{Besicovitch1938,
Marstrand1954, Federer1969, Falconer1985, Mattila1995, MattilaParamonov1995}
for some classical results. For more recent results and references, see
\cite{MeraMoran2001, Lorent2003, MeraMoranPreissZajicek2003, Suomala2005a, Suomala2008, OrponenSahlsten2012} for lower densities and connections to upper porosity and
\cite{KaenmakiSuomala2004, KaenmakiSuomala2008, CsornyeiKaenmakiRajalaSuomala2010, Kaenmaki2010, KaenmakiRajalaSuomala2010a, SahlstenShmerkinSuomala2011} for upper conical density results. We remark that if $\ldimloc(\mu,x)=\udimloc(\mu,x)$ for $\mu$-almost all points,
then \eqref{eq:lconicaldim} follows from the previously known upper conical density estimates. See e.g.\ \cite{CsornyeiKaenmakiRajalaSuomala2010, KaenmakiRajalaSuomala2010a}.

We finish this introduction by fixing some notation. We let $B(x,r)$ denote the closed ball centred at $x\in\R^n$ with radius $r>0$.
Let $n \in \N$, $m \in \{ 0,\ldots,n-1 \}$, and $G(n,n-m)$ be the space of all $(n-m)$-dimensional linear subspaces of $\R^n$. The unit sphere of $\R^n$ is denoted by $S^{n-1}$. For $x \in \R^n$, $\theta \in S^{n-1}$, $V \in G(n,n-m)$, and $0 \le \alpha \le 1$, we set
\begin{align*}
  H(x,\theta,\alpha) &= \{ y \in \R^n : (y-x) \cdot \theta > \alpha|y-x| \}, \\
  X(x,V,\alpha) &= \{ y \in \R^n : \dist(y-x,V) < \alpha|y-x| \}.
\end{align*}
If $\alpha$ is small, then the cone $X(x,V,\alpha)$ is a narrow cone around the translated plane $V$
whereas $H(x,\theta,\alpha)$ is almost a half-space. We write $H(x,\theta)$ for the open half-space $H(x,\theta,0)$.

We will exclusively work with nontrivial Borel regular (outer) measures defined on all subsets of $\R^n$ so that bounded sets have finite measure. For simplicity, we call them just \emph{measures}.
The support of a measure $\mu$, denoted by $\spt(\mu)$, is the smallest closed subset of $\R^n$ with full $\mu$-measure.

Self-similar sets will be referred frequently. The following notation is used in connection to such sets. Let $\kappa \ge 2$ and assume that for each $i \in \{ 1,\ldots,\kappa \}$ there is a mapping $f_i \colon \R^n \to \R^n$ and a constant $0<r_i<1$ so that $|f_i(x)-f_i(y)| = r_i|x-y|$ for all $x,y \in \R^n$. The unique nonempty compact set $E$ satisfying $E = \bigcup_{i=1}^\kappa f_i(E)$ is the \emph{self-similar set}. An \emph{open set condition} is satisfied if there exists a nonempty open set $V$ so that $\bigcup_{i=1}^\kappa f_i(V) \subset V$ with pairwise disjoint union.

Let $\Sigma = \{ 1,\ldots,\kappa \}^\N$, $\Sigma_n = \{ 1,\ldots,\kappa \}^n$, and $\Sigma_* = \{ \varnothing \} \cup \bigcup_{n \in \N} \Sigma_n$. Denote by $|\iii|$ the length of a word $\iii \in \Sigma_* \cup \Sigma$ and if $|\iii| \ge n$, we let $\iii|_n = i_1 \cdots i_n$. Let $\pi$ be the natural projection $\pi\colon\Sigma\to E$ defined by the relation $\{ \pi(\iii) \} = \bigcap_{n \in \N} f_{\iii|_n}(E)$. We also denote by $E_\iii = f_\iii(E) = \pi([\iii])$ the projection of the cylinder set $[\iii] = \{ \iii\jjj : \jjj \in \Sigma \}$ for all $\iii\in\Sigma_*$. Here $f_\iii = f_{i_1} \circ \cdots \circ f_{i_n}$ for all $\iii = i_1\cdots i_n \in \Sigma_n$.

The measure $\nu$ on $\Sigma$ obtained from a probability vector $(p_1,\ldots,p_\kappa)$ by setting $\nu([\iii]) = p_\iii = p_{i_1} \cdots p_{i_n}$ for all $\iii \in \Sigma_n$ is called \emph{Bernoulli measure} and its projection $\mu = \pi\nu$ on $E$ is the \emph{self-similar measure}. If $t \ge 0$ is such that $\sum_{i=1}^\kappa r_i^t = 1$, then the self-similar measure obtained from $(r_1^t,\ldots,r_\kappa^t)$ is called \emph{natural measure}. It is well known that if the open set condition is satisfied, then the natural measure is comparable to $\HH^t|_E$, where $\HH^t$ is the $t$-dimensional Hausdorff measure.

\section{Dimension of general measures on large cones}

The main result in this section is the following theorem.

\begin{theorem} \label{thm:lower_density}
  Let $f \colon (0,1) \to \R$ be an increasing function such that
  \begin{equation} \label{eq:int}
    \int_0^1 \frac{f(t)}{t} \,dt < \infty
  \end{equation}
  and let $\mu$ be a measure on $\R^n$.

  (1) If $\theta \in S^{n-1}$, then
  \begin{equation*}
    \liminf_{r \downarrow 0} \frac{\mu\bigl( B(x,r) \setminus H(x,\theta) \bigr)}{f(r) \mu\bigl( B(x,r) \bigr)} \ge 1
  \end{equation*}
  for $\mu$-almost all $x \in \R^n$.

  (2) If $0 < \alpha \le 1$, then
  \begin{equation*}
    \liminf_{r \downarrow 0} \inf_{\theta \in S^{n-1}} \frac{\mu\bigl( B(x,r) \setminus H(x,\theta,\alpha) \bigr)}{f(r) \mu\bigl( B(x,r) \bigr)} \ge 1
  \end{equation*}
  for $\mu$-almost all $x \in \R^n$.
\end{theorem}

As a corollary, we obtain the local dimension formula for large cones that was mentioned in the introduction.

\begin{corollary} \label{thm:large_angle}
  Suppose $\mu$ is a measure on $\R^n$. If $\theta \in S^{n-1}$, then
  \begin{align*}
    \udimloc(\mu,x) &= \limsup_{r \downarrow 0} \frac{\log \mu\bigl( B(x,r) \setminus H(x,\theta) \bigr)}{\log r}, \\
    \ldimloc(\mu,x) &= \liminf_{r \downarrow 0} \frac{\log \mu\bigl( B(x,r) \setminus H(x,\theta) \bigr)}{\log r}
  \end{align*}
  for $\mu$-almost all $x \in \R^n$. Moreover, if $0 < \alpha \le 1$, then
  \begin{align*}
    \udimloc(\mu,x) &= \limsup_{r \downarrow 0} \sup_{\theta \in S^{n-1}} \frac{\log  \mu\bigl( B(x,r) \setminus H(x,\theta,\alpha) \bigr)}{\log r}, \\
    \ldimloc(\mu,x) &= \liminf_{r \downarrow 0} \sup_{\theta \in S^{n-1}} \frac{\log  \mu\bigl( B(x,r) \setminus H(x,\theta,\alpha) \bigr)}{\log r}
  \end{align*}
  for $\mu$-almost all $x \in \R^n$.
\end{corollary}

\begin{remark} \label{rem:lower_density}
  (1) If $s>1$, then the function $f(t)=|\log t|^{-s}$ satisfies the condition \eqref{eq:int}. Observe also that the condition \eqref{eq:int} is equivalent to $\sum_{i=1}^\infty f(2^{-i}) < \infty$.

  (2) The condition \eqref{eq:int} is sharp in the sense that if $\int_0^1 t^{-1}f(t) dt=\infty$, then there is a measure $\mu$ on $\R$ such that
  \[
    \liminf_{r\downarrow 0}\frac{\mu([x,x+r])}{f(r)\mu([x-r,x+r])}=0
  \]
  for $\mu$-almost all $x\in\R$. This is proved in Proposition \ref{ex:sharp}.

  (3) If $\alpha(t)>0$ is an increasing function with $\lim_{t\downarrow 0}\alpha(t)=0$, then, without any changes in the proof, Theorem \ref{thm:lower_density}(2) can be strengthened to
  \[
    \liminf_{r \downarrow 0} \inf_{\theta \in S^{n-1}} \frac{\mu\bigl(B(x,r) \setminus H(x,\theta,\alpha(r)) \bigr)}{f(r)\mu\bigl( B(x,r) \bigr)}
  \]
  provided that $f$ satisfies $\int_{0}^1 t^{-1}c(n,\alpha(t))^{-1} f(t) dt<\infty$ and $c(n,\alpha(t))$ is as in Lemma \ref{thm:covering}(2).

  (4) Compactness of $S^{n-1}$ implies that Theorem \ref{thm:lower_density}(2) is equivalent to the claim according to which for $\mu$-almost all $x \in \R^n$ we have
  \begin{equation*}
    \liminf_{r \downarrow 0} \frac{\mu\bigl( B(x,r) \setminus H(x,\theta,\alpha) \bigr)}{f(r) \mu\bigl( B(x,r) \bigr)} \ge 1
  \end{equation*}
  simultaneously for all $\theta \in S^{n-1}$.

  (5) It is clear that in Theorem \ref{thm:lower_density}(2) the cones $H(x,\theta,\alpha)$ cannot be replaced by $H(x,\theta)$. For example, consider the length measure on a circle.

  (6) In the view of Corollary \ref{thm:large_angle}, local dimensions of $\mu$ on $\R$ can be calculated via one-sided balls for $\mu$-almost all points. Falconer \cite{Falconer2004} has shown that for the natural measure $\mu$ on the $\tfrac13$-Cantor set $C\subset\R$, the exceptional set $\{ x \in \R : \udimloc(\mu,x) < \limsup_{r \downarrow 0} \log\mu([x,x+r])/\log r \}$ can have as large Hausdorff dimension as $C$. However, any self-similar measure $\mu$ on $C$ satisfies $\ldimloc(\mu,x) = \liminf_{r \downarrow 0} \log\mu([x,x+r])/\log r$ except for at most countably many points; see Proposition \ref{thm:falconer}. For the natural measure, this is shown in \cite{Falconer2004}.
\end{remark}

Before proving Theorem \ref{thm:lower_density}, we exhibit a covering lemma suitable for our purposes. Its proof is based on simple geometric inspections (cf.\ \cite[Theorem 3.1]{CsornyeiKaenmakiRajalaSuomala2010}).

\begin{lemma} \label{thm:covering}
  Suppose that $\mu$ is a measure on $\R^n$, $A \subset \R^n$ is a bounded 
  set, $0<R<\infty$, and $R \le r_x \le 2R$ for all $x \in A$.

  (1) If $\theta \in S^{n-1}$, then there exists a finite set $F \subset A$ so that the collection $\{ B(x,r_x) \}_{x \in F}$ is pairwise disjoint and
  \begin{equation*}
    \sum_{x \in F} \mu \bigl( B(x,r_x) \setminus H(x,\theta) \bigr) \ge c\mu(A),
  \end{equation*}
  where $c=c(n)>0$ is a constant that depends only on $n$.

  (2) If $\theta_x \in S^{n-1}$ for all $x \in A$ and $0 < \alpha \le 1$, then there exists a finite set $F \subset A$ so that the collection $\{ B(x,r_x) \}_{x \in F}$ is pairwise disjoint and
  \begin{equation*}
    \sum_{x \in F} \mu \bigl( B(x,r_x) \setminus H(x,\theta_x,\alpha) \bigr) \ge c\mu(A),
  \end{equation*}
  where $c=c(n,\alpha)>0$ is a constant that depends only on $n$ and $\alpha$.
\end{lemma}

\begin{proof}
  (1) Let $\{ B(x,R/4) \}_{x \in F_0}$ be a maximal packing of $A$. Thus $F_0 \subset A$ is a finite set and $A \subset \bigcup_{x \in F_0} B(x,R/2)$. A simple volume argument implies that there exists a positive constant $C=C(n)$ so that  for each $y\in \R^n$, the ball $B(y, 5R)$ intersects at most $C$ balls  $B(x,5R)$ with $x\in F_0$. This in turn implies that we may decompose $F_0$ into $C$ subsets such that the points in each subset have mutual distance at least $5R$. For some $F_1$ in this decomposition, it then follows that
  \begin{equation*}
    \mu\biggl(A\cap \Big( \bigcup_{x \in F_1} B(x,R/2)\Big) \biggr) \ge \mu(A)/C
  \end{equation*}
  and $|x-y| \ge 5R$ for all $x,y \in F_1$ with $x \neq y$.

  For $x \in F_1$, let $t = \sup\{ y \cdot \theta : y\in A\cap B(x,R/2) \}$. If $y_n\in A\cap B(x,R/2)$ so that $y_n\cdot\theta\rightarrow t$ as $n \to \infty$, then it follows that $\mu\bigl( A\cap B(x,R/2) \setminus H(y_n,\theta) \bigr) \rightarrow \mu\bigl( A\cap B(x,R/2) \bigr)$. Recall that $H(y_n,\theta)$ is an open half-space. In particular, this implies that we can pick $y_x \in A \cap B(x,R/2)$ for which
  \[
    \mu\bigl( B(x,R/2) \setminus H(y_x,\theta) \bigr) \ge \tfrac12 \mu\bigl(A \cap B(x,R/2) \bigr).
  \]
  Let $F = \{ y_x : x \in F_1 \}$. Since the collection $\{ B(y_x,r_{y_x}) \}_{x \in F_1}$ is pairwise disjoint, we arrive at
  \begin{align*}
    \sum_{x \in F} \mu \bigl( B(x,r_x) \setminus H(x,\theta) \bigr) &\ge \sum_{x \in F_1} \mu \bigl( B(x,R/2) \setminus H(y_x,\theta) \bigr) \\
    &\ge \frac12 \sum_{x \in F_1} \mu\bigl(A \cap B(x,R/2) \bigr)\\
   &= \frac12 \mu\biggl(A\cap \Big( \bigcup_{x \in F_1} B(x,R/2)\Big) \biggr) \ge \mu(A)/(2C)
  \end{align*}
  finishing the proof of (1).

  (2) Choose $\roo=\roo(\alpha)>0$ so that $H(0,\theta,\alpha) \subset H(0,\zeta)$ for all $\zeta \in S^{n-1}$ and $\theta \in S^{n-1} \cap B(\zeta,\roo)$. Since $S^{n-1}$ is compact, we find $M=M(n,\roo) \in \N$ and $\zeta_1,\ldots,\zeta_M \in S^{n-1}$ such that $S^{n-1} \subset \bigcup_{j=1}^M B(\zeta_j,\roo)$. Thus there is $j_0 \in \{ 1,\ldots,M \}$ so that $\mu(A') \ge \mu(A)/M$, where $A' = \{ x \in A : \theta_x \in B(\zeta_{j_0}, \roo) \}$.
  Observe that
  \[
    B(x,r_x) \setminus H(x,\theta_x,\alpha) \supset B(x,r_x) \setminus H(x,\zeta_{j_0})
  \]
  for all $x \in A'$. Applying now (1) to the set $A'$ and $\zeta_{j_0} \in S^{n-1}$ yields the claim.
\end{proof}

\begin{proof}[Proof of Theorem \ref{thm:lower_density}]
Without loss of generality we may assume that $\mu$ has bounded support. Let $\theta\in S^{n-1}$ and define
\begin{equation*}
  A_i = \bigl\{ x \in \R^n : \frac{\mu\bigl( B(x,r_x) \setminus H(x,\theta) \bigr)}{f\bigl( r_x \bigr) \mu\bigl( B(x,r_x) \bigr)} < 1 \text{ for some }2^{-i-1}\le r_x<2^{-i} \bigr\}
\end{equation*}
for all $i \in \N$.
Applying Lemma \ref{thm:covering}(1), we find finite sets $F_i\subset A_i$ such that
$\{B(x,r_x)\}_{x\in F_i}$ are pairwise disjoint and $\sum_{x\in F_i} \mu\bigl( B(x,r_x) \setminus H(x,\theta) \bigr) \ge c\mu(A_i)$, where $c=c(n)>0$ is the constant from Lemma \ref{thm:covering}(1). Together with the definition of $A_i$ this implies
\[
  \mu(A_i) \le c^{-1}\sum_{x\in F_i} \mu\bigl( B(x,r_x) \setminus H(x,\theta) \bigr) \le c^{-1}\mu(\R^n) f(2^{-i}).
\]
Since $\sum_{i=1}^\infty f(2^{-i}) < \infty$, we have $\sum_{i=1}^\infty \mu(A_i)<\infty$. The first claim is now proved since $\mu$-almost all $x\in\R^n$ belong to only finitely many sets $A_i$ by the Borel-Cantelli lemma.

The second claim is proved in the same way by considering
\begin{align*}
  A_i = \{ x \in \R^n : \;&\mu\bigl( B(x,r_x) \setminus H(x,\theta,\alpha) \bigr) < f\bigl( r_x \bigr) \mu\bigl( B(x,r_x) \bigr) \\
  &\text{for some } 2^{-i-1} \le r_x < 2^{-i} \text{ and } \theta \in S^{n-1}\},
\end{align*}
and using Lemma \ref{thm:covering}(2).
\end{proof}

The following example verifies the sharpness of the integrability condition in Theorem \ref{thm:lower_density}.

\begin{proposition} \label{ex:sharp}
  Let $C\subset[0,1]$ be the $\tfrac13$-Cantor set and $\mu$ its natural measure. If $f \colon (0,1) \to \R$ is an increasing function such that
  \begin{equation} \label{eq:fassumption}
    \int_0^1 \frac{f(t)}{t} \,dt = \infty,
  \end{equation}
  then
  \begin{equation} \label{eq:claimA}
    \liminf_{r \downarrow 0} \frac{\mu([x,x+r])}{f(r)\mu([x-r,x+r])} = 0
  \end{equation}
  for $\mu$-almost all $x \in C$.
\end{proposition}

We give two alternative proofs for this proposition.

\begin{proof}[Constructive proof]
  It suffices to show that
  \begin{equation}\label{eq:claimB}
    \liminf_{r \downarrow 0} \frac{\mu([x,x+r])}{f(r)\mu([x-r,x+r])} \le 1
  \end{equation}
  for $\mu$-almost all $x\in C$ since by scaling the function $f$, this implies \eqref{eq:claimA}.

  Observe that if $\mathcal{I}_k$ is the collection of the $2^k$ construction intervals of $C$ of length $3^{-k}$, then $\mu(I)=2^{-k}$ for each $I\in\mathcal{I}_k$. If $I = [a,c] \in \II_k$, we choose $a<b<c$ such that $\mu([b,c]) = f(3^{-k})\mu(I)$ and denote $I_+ = [b,c]$. Moreover, we set $E_k = C \setminus \bigcup_{I \in \II_k} I_+$. If $x\in C\setminus\bigcup_{N=1}^\infty\bigcap_{k\ge N}E_k$, then \eqref{eq:claimB} holds. Thus it remains to show that
  \begin{equation} \label{eq:claimC}
    \mu\biggl( \bigcap_{k \ge N} F_k \biggr)=0
  \end{equation}
  for all $N \in \N$ where $F_k = \bigcap_{n=N}^{k} E_n$.

  Given $L \in \N$ and $I \in \II_k$, we let $I_L$ be the leftmost sub-construction interval of $I$ of size $3^{-k-L}$. Let us consider the following condition:
  \begin{itemize}
    \item[(H)] There exist $\eps > 0$ and $L \in \N$ such that for all $k \ge N$ there are at least $\eps 2^{k}$ intervals $I \in \II_k$ so that $I_L \cap C \subset F_k$.
  \end{itemize}
  Let us first assume that the condition (H) does not hold. Then, for any $\eps > 0$ and $L \in \N$, we find $k \ge N$ such that for at least $(1-\eps) 2^k$ intervals $I \in \II_k$ we have $I_L \cap C \not\subset F_k$. For any such an interval $I$ there
is $I' \in \bigcup_{n=N}^k \II_n$ such that $I_L \cap I'_+ \neq \emptyset$. As the right endpoint of $I'_+$ is the right endpoint of $I'$, it follows that $I_L \subset I'$. Thus $I \cap F_k \subset I_L$ and, consequently, $\mu(F_k \cap I) \le 2^{-L}\mu(I)$. Putting these estimates together yields
  \[
    \mu(F_k) \le \eps + (1-\eps) 2^{-L}
  \]
  and as $\eps$ and $L$ were arbitrary this implies \eqref{eq:claimC}.

  Now we assume that the condition (H) holds. It follows that $$\mu(F_{k+L}) \le \mu(F_k)\bigl( 1-\eps 2^{-L} f(3^{-k-L}) \bigr)$$ for all $k \ge N$. Using this inductively, we get
  \[
    \mu(F_{N+nL}) \le \prod_{k=1}^n \bigl( 1-\varepsilon 2^{-L} f(3^{-N-kL}) \bigr).
  \]
  Since the condition \eqref{eq:fassumption} is equivalent to $\prod_{k=1}^\infty \bigl( 1-\eps 2^{-L} f(3^{-N-kL}) \bigr) = 0$, this completes the proof.
\end{proof}

The proposition can also be deduced from the classical result of Erd{\H{o}}s and R{\'e}v{\'e}sz (\cite[Theorem 7.2]{Revesz2005}) on the longest length of consequtive zeros appearing in a random sequence of digits. We provide the details below, since we are going to use similar arguments in \S \ref{selfsim} below in connection to self-similar measures.

\begin{proof}[Probabilistic proof]
If $f_1(x) = \tfrac13 x$ and $f_2(x) = \tfrac13 x + \tfrac23$ are the mappings that generate $C$, then in the projection mapping $\pi \colon \{ 1,2 \}^\N \rightarrow C$ the symbol $1$ corresponds to ``left'' and $2$ to ``right''. Given $\iii = i_1i_2 \cdots \in \{ 1,2 \}^\N$, let $\Gamma_n(\iii)$ be the number of consecutive $2$'s in $\iii$ appearing after $\iii|_n = i_1\cdots i_n \in \{1,2\}^n$. Then, for $x=\pi(\iii)$, we have
\begin{equation}\label{eq:base3}
  \mu([x,x+3^{-n}])\le 2^{-\Gamma_n(\iii)}\mu([x-3^{-n},x+3^{-n}]).
\end{equation}

Let $(a_n)$ be a sequence of positive integers. 
The behaviour of $\Gamma_n(\iii)$ was characterised by Erd{\H{o}}s and R{\'e}v{\'e}sz \cite{ErdosRevesz1977}. 
They showed that for $\nu$-almost every $\iii\in\{1,2\}^\N$,
\begin{align}\label{eq:io}
  \Gamma_n(\iii) > a_n \text{ infinitely often}
\end{align}
if and only if $\sum_{n=1}^\infty 2^{-a_n} = \infty$.
Since \eqref{eq:fassumption} implies $\sum_{n=1}^\infty f(3^{-n}) = \infty$, we apply \eqref{eq:io} with $a_n=-\log_2 f(3^{-n})$. Hence, combining this with \eqref{eq:base3}, it follows that for $\mu$-almost all $x=\pi(\iii)$,
\[
  \mu([x,x+3^{-n}])\le 2^{-\Gamma_n(\iii)}\mu([x-3^{-n},x+3^{-n}])\le f(3^{-n})\mu([x-3^{-n},x+3^{-n}])
\]
for infinitely many $n$.
\end{proof}

We finish this section by verifying the claim in Remark \ref{rem:lower_density}(6).

\begin{proposition} \label{thm:falconer}
  If $\mu$ is a self-similar measure on the $\tfrac13$-Cantor set $C$, then
  \begin{equation*}
    \ldimloc(\mu,x) = \liminf_{r \downarrow 0} \frac{\log\mu([x,x+r])}{\log r}
  \end{equation*}
  except for at most countably many points $x \in \R$.
\end{proposition}

\begin{proof}
  It is well known that $\ldimloc\bigl( \mu,\pi(\iii) \bigr) = \liminf_{n \to \infty} -\log\mu(E_{\iii|_n})/(n\log 3)$ for all $\iii \in \{ 1,2 \}^\N$. Observe that $-\log\mu(E_{\iii|_n})/n$ is monotone on each block in which $i_n$ is constant. Thus for all $\iii \in \{ 1,2 \}^\N$ for which there are infinitely many such kind of finite blocks (i.e.\ for all $x = \pi(\iii)$ except at the end-points of the construction intervals) there is a sequence $(n_k)$ such that
  \begin{equation*}
    \lim_{k \to \infty} \frac{-\log\mu(E_{\iii|_{n_k}})}{n_k\log 3} = \ldimloc\bigl( \mu,\pi(\iii) \bigr)
  \end{equation*}
  and $i_{n_k} \ne i_{n_k+1}$. The claim follows since now $\mu([\pi(\iii),3^{-(n_k-1)}])$ is comparable to $\mu(E_{\iii|_{n_k}})$ for all $k \in \N$ (it is possible that $\mu([\pi(\iii),3^{-n_k}])$ is not comparable to $\mu(E_{\iii|_{n_k}})$).
\end{proof}

\section{Dimension of general measures on narrow cones} 

The arguments in this section are based on the standard techniques used to obtain conical density estimates for purely unrectifiable measures. We refer to \cite[\S 15]{Mattila1995} for the basic properties of rectifiable sets.

For $x \in \R^n$, $V \in G(n,n-m)$, $0 \le \alpha \le 1$, and $\beta \ge 1$, we define a twisted cone by setting
\[
  X^\beta(x,V,\alpha) = \{ y \in \R^n : \dist(y-x,V) < \alpha|y-x|^\beta \}.
\]
The following lemma is needed also in \S \ref{selfsim}.

\begin{lemma} \label{thm:pertti}
  Let $\mu$ be a measure on $\R^n$, $A \subset \R^n$, $V \in G(n,n-m)$, $\theta \in S^{n-1}$, $0<\alpha\le 1$, $\beta \ge 1$, and $r>0$. If
  \begin{equation*}
    \mu\bigl( B(x,r) \cap X^\beta(x,V,\alpha) \setminus H(x,\theta,\alpha) \bigr) = 0
  \end{equation*}
  for $\mu$-almost all $x \in A$,
  then $A \cap \spt(\mu)$ is contained in a countable union of images of $\tfrac{1}{\beta}$-H\"older continuous maps $A \cap V^\bot \to \R^n$ and thus, $\dimp(A \cap \spt(\mu)) \le \beta m$. Furthermore, if $\beta=1$, then $A \cap \spt(\mu)$ is $m$-rectifiable.
\end{lemma}

\begin{proof}
  The proof is essentially identical to  that of \cite[Lemma 15.13]{Mattila1995}. One has to just notice that if $x,y \in A \cap \spt(\mu)$ such that $|y-x|<r$ and $|\proj_{V^\bot}(y-x)| < \alpha|y-x|^\beta$, then not only $y \in B(x,r) \cap X^\beta(x,V,\alpha) \cap H(x,\theta,\alpha)$ but also $x \in B(y,r) \cap X^\beta(y,V,\alpha) \setminus H(y,\theta,\alpha)$. Thus $|\proj_{V^\bot}(y-x)| \ge \alpha|y-x|^\beta$ and $(\proj_{V^\bot}|_A)^{-1}$ is the desired mapping. 
\end{proof}

For the lower local dimension in twisted cones, we have the following estimate.

\begin{theorem}\label{thm:twisted}
  If $\mu$ is a measure on $\R^n$, $V\in G(n,n-m)$, $0<\alpha\le 1$, and $\beta\ge 1$, then
  \begin{equation*}
    \liminf_{r \downarrow 0} \frac{\log\mu\bigl( B(x,r) \cap X^\beta(x,V,\alpha) \bigr)}{\log r}\le m(\beta-1)+\ldimloc(\mu,x)
  \end{equation*}
  for $\mu$-almost all $x$ with $\udimloc(\mu,x)>\beta m$.
\end{theorem}

%

\begin{proof}
Assume to the contrary that there are $r_0>0$, $\beta<\gamma<n/m$, $s>\beta m$, and a Borel set $E\subset\R^n$ with $\mu(E)>0$ such that $\ldimloc(\mu,x)<s$ and
\begin{equation}\label{sd}
  \mu\bigl( B(x,r) \cap X^\beta(x,V,\alpha) \bigr) < r^{m(\gamma-1)+s}
\end{equation}
for all $x\in E$ and $0<r<r_0$.

Since $\mu$-almost all points of $E$ are density points (\cite[Corollary 2.14]{Mattila1995}) and $\ldimloc(\mu,x) < s$ for all $x \in E$, there are $x_0 \in E$ and arbitrary small $0<r<r_0$ so that
\begin{equation} \label{eq:dims}
  \mu\bigl( E \cap B(x_0,r) \bigr) > 2^{s-m+1} \cdot 16^n \cdot 10^m \cdot \alpha^{-m} r^s.
\end{equation}
Fix such a radius so that $r^{\gamma-\beta} < 4^{-\beta}$.

For each $x \in E \cap B(x_0,r/2)$, we define $$h(x) = \sup\{|y-x| : y \in E \cap B(x_0,r) \cap X^\gamma(x,V,\alpha) \}.$$ Since $\udimloc(\mu|_E,x) > \beta m$ for $\mu$-almost every $x \in E$, Lemma \ref{thm:pertti} implies that $\mu|_E\bigl( B(x,r/2) \cap X^\beta(x,V,\alpha) \bigr) > 0$, and, consequently, $0<h(x)<2r$ for $\mu$-almost all $x \in E \cap B(x_0,r/2)$.

Moreover, by simple geometric inspections, we find that for
\[
  C(x) = B(x_0,r) \cap \proj_{V^\perp}^{-1}\bigl( \proj_{V^\perp}B(x,\alpha h(x)^\gamma) \bigr)
\]
we have
\begin{equation}\label{strips}
  C(x) \subset \bigl( B(x,4h(x))\cap X^\beta(x,V,\alpha) \bigr) \cup \bigl( B(y,4 h(x))\cap X^\beta(y,V,\alpha) \bigr)
\end{equation}
for some $y \in E \cap B(x_0,r) \cap X^{\gamma}(x,V,\alpha)$. See Figure \ref{fig:twisted} and recall the proofs of \cite[Lemma 15.14]{Mattila1995} and \cite[Theorem 5.1]{Kaenmaki2010}.
  \begin{figure}
    \psfrag{x}{$x$}
    \psfrag{y}{$y$}
    \psfrag{V}{$V$}
    \begin{center}
      \includegraphics[width=0.9\textwidth]{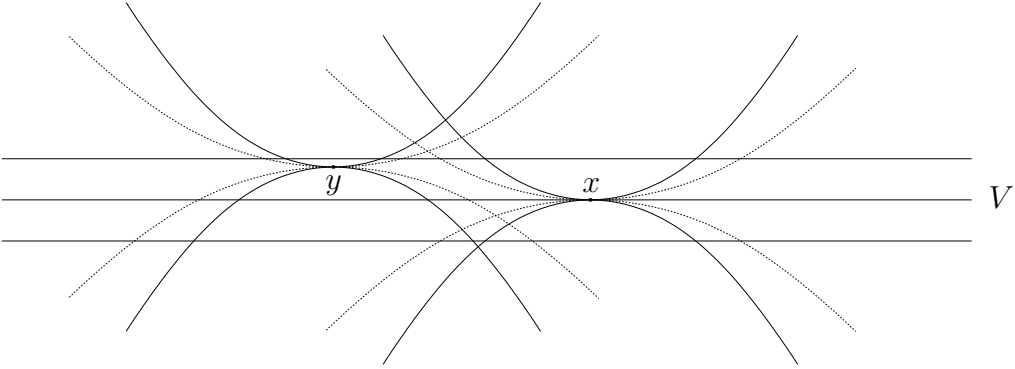}
    \end{center}
    \caption{The choice $r^{\gamma-\beta} < 4^{-\beta}$ guarantees that $C(x)$ is contained in the union of $X^\beta(x,V,\alpha)$ and $X^\beta(y,V,\alpha)$ illustrated by the solid curves in the picture.}
    \label{fig:twisted}
  \end{figure}

By the $5r$-covering theorem, we find a countable collection of pairwise disjoint balls $\{ \proj_{V^\bot}B(x_i,\alpha h(x_i)^\gamma/5) \}_i$ with $x_i \in E \cap B(x_0,r)$ so that
\begin{equation*}
  \bigcup_{h(x)>0} \proj_{V^\bot}B\bigl( x,\alpha h(x)^\gamma \bigr) \subset \bigcup_i \proj_{V^\bot}B\bigl( x_i,\alpha h(x_i)^\gamma \bigr).
\end{equation*}
Observe that we have
\begin{align*}
\sum_i 2^m \alpha^m h(x_i)^{\gamma m}/5^m &= \sum_i \HH^m \bigl( \proj_{V^\bot}B(x_i,\alpha h(x_i)^\gamma/5) \bigr) \\
 &\le \HH^m\bigl( \proj_{V^\bot}B(x_0,2r) \bigr) = 2^m(2r)^m
 \end{align*}
  and $\mu\bigl( E \cap B(x_0,r) \bigr) \le \sum_i \mu\bigl( C(x_i) \bigr)$. Here we assumed that the Hausdorff measure is normalized. Combining these estimates with \eqref{strips}, \eqref{sd}, and the fact that $h(x_i)^{s-m} \le (2r)^{s-m}$ gives
\begin{align*}
\mu\bigl( E \cap B(x_0,r) \bigr) &\le \sum_i\mu\bigl( C(x_i) \bigr) \le 2 \cdot 4^{\gamma m+s-m}\sum_i h(x_i)^{\gamma m+s-m}\\
&\le 2 \cdot 16^n\sum_i h(x_i)^{s-m}h(x_i)^{\gamma m} < 2^{s-m+1} \cdot 16^n \cdot 10^m \alpha^{-m} r^s,
\end{align*}
a contradiction with \eqref{eq:dims}.
\end{proof}

\begin{remark}
The upper bound of Theorem \ref{thm:twisted} is seen to be sharp by considering Hausdorff measures on $V^\perp\times C$, for self-similar Cantor sets $C\subset V$.
%
%
%
\end{remark}

For the ordinary cones we get the following stronger result. This should be compared to Theorem \ref{thm:twisted} when $\beta\downarrow 1$.
We formulate the result for purely unrectifiable measures to cover also the case $\udimloc(\mu,x)=m$. A measure $\mu$ on $\R^n$ is called \emph{purely $m$-unrectifiable} if $\mu(A)=0$ for all $m$-rectifiable sets $A \subset \R^n$. Observe that $\mu$ restricted to the set $\{ x \in \R^n : \udimloc(\mu,x)>m \}$ is purely $m$-unrectifiable.

\begin{theorem} \label{thm:lower_dim}
  If $\mu$ is a measure on $\R^n$ and $0 < \alpha \le 1$, then
  \begin{equation*}
  \ldimloc(\mu,x) = \liminf_{r \downarrow 0} \sup_{V \in G(n,n-m)} \frac{\log\mu\bigl( B(x,r) \cap X(x,V,\alpha) \bigr)}{\log r}
  \end{equation*}
  for $\mu$-almost all $x \in \R^n$ with $\udimloc(\mu,x) \ge m$ provided that $\mu$ is purely $m$-unrectifiable.
\end{theorem}

\begin{proof}
  By a simple compactness argument (see \cite[Remark 4.4]{CsornyeiKaenmakiRajalaSuomala2010}) it suffices to show the claim for a fixed $V \in G(n,n-m)$. After this observation, the proof continues as the proof of Theorem \ref{thm:twisted}.
\end{proof}

\begin{remark} \label{rem:lower_dim}
  (1) Observe that for $\mu$-almost all $x \in \R^d$ with $$\udimloc(\mu,x) = \ldimloc(\mu,x) > m,$$ the claim of Theorem \ref{thm:lower_dim} follows from \cite[Theorem 5.1]{KaenmakiRajalaSuomala2010a}, and, in fact, we may replace the cone $X(x,V,\alpha)$ by a non-symmetric cone $X(x,V,\alpha) \setminus H(x,\theta,\alpha)$ and take the supremum over all $V \in G(n,n-m)$ and $\theta \in S^{n-1}$. We do not know whether Theorem \ref{thm:lower_dim} holds for non-symmetric cones in general. This question is interesting already in the case $m=n-1$. Recall also \cite[Theorem 5.1]{Kaenmaki2010}.

  (2) Recall from \cite[Example 5.5]{CsornyeiKaenmakiRajalaSuomala2010}, that for $\ell\in G(2,1)$, there is a measure $\mu$ on $\R^2$ so that $\mu$ is purely $1$-unrectifiable and for every $0<\alpha<1$
  \begin{equation*}
    \lim_{r \downarrow 0} \frac{\mu\bigl( B(x,r) \cap X(x,\ell,\alpha) \bigr)}{\mu\bigl( B(x,r) \bigr)} = 0
  \end{equation*}
  for $\mu$-almost all $x \in \R^2$. It is therefore interesting to ask if it is possible to obtain finer information on the ratio $\mu\bigl( B(x,r) \cap X(x,V,\alpha) \bigr)/\mu\bigl( B(x,r) \bigr)$ for arbitrary small $r>0$. See also Proposition \ref{thm:example}.
\end{remark}

If the dimension of a measure is small it is often the case that the limit superior (resp.\ inferior) of $\log\mu\bigl( B(x,r) \cap X(x,V,\alpha) \bigr)/\log r$ is strictly larger than the upper (resp.\ lower) local dimension of $\mu$ at $x$. For example, this is the case when $\spt(\mu)$ is contained in a rectifiable curve. Perhaps surprisingly, this behaviour is possible also if the local dimension of the measure is large.

Besides exhibiting the previously described phenomenon, the example in the following proposition shows that Theorem \ref{thm:lower_dim} cannot hold for the upper local dimension. The observation that lower conical dimensions are often more regular (or less ``multifractal'') than the upper ones was found in the one dimensional situation already by Falconer \cite{Falconer2004}; see Remark \ref{rem:lower_density}(6).


\begin{proposition} \label{thm:example}
  If $\ell \in G(n,1)$ and $0<\alpha<1$, then for every $1<s<t<2$ there is a measure $\mu$ on $\R^2$ such that $\ldimloc(\mu,x)=s$ and
  \begin{equation*}
    \limsup_{r \downarrow 0} \frac{\log\mu\bigl( B(x,r) \cap X(x,\ell,\alpha) \bigr)}{\log r} = s\,\frac{t-1}{s-1} > t = \udimloc(\mu,x)
  \end{equation*}
  for $\mu$-almost all $x \in \R^2$.
\end{proposition}

\begin{proof}
  We may assume that $\ell$ is the $y$-axis. Let us denote the side-length of a square $Q$ by $|Q|$. In this proof, all the squares have sides parallel to the coordinate axes. We construct the measure $\mu$ using the mass distribution principle and the following two operations:
  \begin{itemize}
    \item[(1)] Suppose $Q \subset \R^2$ is a square with $|Q| \le 1$ and $\mu(Q) \ge |Q|^t$. Let $n$ be an integer such that $n-1 \le (|Q|^{-t} \mu(Q))^{1/(2-t)} < n$. Divide $Q$ into $n^2$ subsquares of side-length $|Q|/n$ and assign a measure
$\mu(Q)/n^2$ to each of these subsquares.
    \item[(2)] Suppose $Q \subset \R^2$ is a square with $|Q| \le 1$ and $\mu(Q) \le |Q|^s$. Let $n$ be an integer such that $n-1 \le (|Q|^{s}/\mu(Q))^{1/(s-1)} < n$. Divide $Q$ into $n^2$ subsquares of side-length $|Q|/n$ and assign a measure $\mu(Q)/n$ to the subsquares in the bottom row and measure $0$ to others.
  \end{itemize}
  We construct $\mu$ by first applying the operation $(1)$ with $Q=[0,1]\times[0,1]$ and $\mu(Q)=1$. Then we apply the operation $(2)$ for each of the $4$ subsquares of side-length $1/2$ and the operation $(1)$ for the squares in their bottom row. We continue in this manner. It is easy to see 
  that for the resulting measure, we have $\udimloc(\mu,x) = t$ and $\ldimloc(\mu,x)=s$ for all $x \in \spt(\mu)$.

  To show the claim on the conical dimension, observe first that $\log\mu\bigl( B(x,r) \cap X(x,\ell,\alpha) \bigr)/\log r$ obtains a local maximum right after the operation (1). Let $Q$ be a square as in the operation (2). Then $X(x,\ell,\alpha)$ intersects at most constant many squares $Q'$ in the bottom row of $Q$ and its horizontal neighbours. Thus the estimate
  \begin{equation*}
    \frac{\log\mu(Q')}{\log|Q|} = \frac{\log\tfrac{1}{n}\mu(Q)}{\log|Q|} \ge \frac{\log\mu(Q)^{1+1/(s-1)}|Q|^{-s/(s-1)}}{\log|Q|} = \frac{s}{s-1}t - \frac{s}{s-1}
  \end{equation*}
  implies the claim.
\end{proof}

\section{Dimension of self-similar measures on narrow cones} 
\label{selfsim}

Finally, we turn our attention to self-similar sets and consider measures on narrow cones around $(n-m)$-planes.

\begin{theorem} \label{thm:selfsimilar}
  Let $\mu$ be a self-similar measure on a self-similar set $E \subset \R^n$ satisfying the open set condition.
  If $\mu$ is purely $m$-unrectifiable and $0 < \alpha \le 1$, then there is $1 < s = s(\mu,n,m,\alpha) < \infty$ so that
  \[
    \liminf_{r\downarrow 0}\inf_{\yli{\theta \in S^{n-1}}{V\in G(n,n-m)}} \frac{\mu\bigl( B(x,r) \cap X(x,V,\alpha) \setminus H(x,\theta,\alpha) \bigr)}{|\log r|^{-s}\mu\bigl( B(x,r) \bigr)} \ge 1
  \]
  for $\mu$-almost all $x \in \R^n$.
\end{theorem}

Again, as a direct corollary, we obtain formula for the local dimensions via narrow cones.

\begin{corollary} \label{cor:selfsimilar}
  Let $\mu$ be a self-similar measure on a self-similar set $E \subset \R^n$ satisfying the open set condition. If $\mu$ is purely $m$-unrectifiable and $0 < \alpha \le 1$, then
  \begin{align*}
    \udimloc(\mu,x) &= \limsup_{r \downarrow 0} \sup_{\yli{\theta \in S^{n-1}}{V\in G(n,n-m)}} \frac{\log \mu\bigl( B(x,r) \cap X(x,V,\alpha) \setminus H(x,\theta,\alpha) \bigr)}{\log r}, \\
    \ldimloc(\mu,x) &= \liminf_{r \downarrow 0} \sup_{\yli{\theta \in S^{n-1}}{V\in G(n,n-m)}} \frac{\log \mu\bigl( B(x,r) \cap X(x,V,\alpha) \setminus H(x,\theta,\alpha) \bigr)}{\log r}
  \end{align*}
  for $\mu$-almost all $x \in \R^n$.
\end{corollary}

\begin{remark}
  (1) As one would expect, self-similar measures behave more regularly than general measures; compare Corollary \ref{cor:selfsimilar} to Proposition \ref{thm:example}. Observe also that there is no lower bound for the local dimension of the self-similar measure.

  (2) If $\mu$ is a self-similar measure, then for $\mu$-almost all $x \in \R^d$ with $\udimloc(\mu,x) > m$, the latter claim of Corollary \ref{cor:selfsimilar} follows from Feng and Hu \cite[Theorem 2.8]{FengHu2009} and Remark \ref{rem:lower_dim}(1) even without assuming the open set condition.

  (3) Mattila \cite{Mattila1982} has shown that a self-similar set $E$ either lies on an $m$-dimensional affine subspace or $\HH^t(E \cap M) = 0$ for every $m$-dimensional $C^1$-submanifold of $\R^n$. Here $t=\dimh(E)$. Further generalizations of this result can be found in \cite{Kaenmaki2003, Kaenmaki2006, BandtKravchenko2011}.

  (4) By inspecting the proof of Theorem \ref{thm:selfsimilar}, one is easily convinced that the result holds also for self-conformal sets.

  (5) An interesting question is whether Theorem \ref{thm:selfsimilar} remains true for every purely $1$-unrectifiable measure. Recall constructions presented in \cite[\S 5.3]{MartinMattila1988}, \cite[\S 5.8]{Preiss1987}, and \cite[Example 5.4]{CsornyeiKaenmakiRajalaSuomala2010}.
\end{remark}

\begin{proof}[Proof of Theorem \ref{thm:selfsimilar}]
  Let $f_1,\ldots,f_\kappa$ be the defining similitudes and $\nu$ the Bernoulli measure on $\Sigma$ for which $\pi\nu = \mu$. Let $0<r_1<\cdots<r_\kappa<1$ be the contraction ratios and $p = \min_{i \in \{ 1,\ldots,\kappa \}} \nu([i]) > 0$ the smallest Bernoulli weight. We show that there are $l \in \N$ and $\hhh \in \Sigma_l$ so that for each $\iii \in \Sigma_*$, $y \in E_{\iii\hhh}$, $V\in G(n,n-m)$, and $\theta \in S^{n-1}$, we have
  \begin{equation}\label{coneincl}
    E_{\iii\jjj} \subset X(y,V,\alpha) \setminus H(y,\theta,\alpha)
  \end{equation}
  for some $\jjj \in \Sigma_l$.

  To prove the above claim, we may assume that $\iii$ above is $\varnothing$ since by self-similarity, the claim is invariant under $f_\iii$ for all $\iii \in \Sigma_*$. By \cite[Remark 4.4]{CsornyeiKaenmakiRajalaSuomala2010} and the compactness of $S^{n-1}$, there are $V_1,\ldots,V_{M_1} \in G(n,n-m)$ and $\theta_1,\ldots,\theta_{M_2} \in S^{n-1}$ such that for any $V\in G(n,n-m)$ and $\theta \in S^{n-1}$ it holds that $X(0,V_i,\alpha/2) \subset X(0,V,\alpha)$ and $H(0,\theta,\alpha) \subset H(0,\theta_j,\alpha/2)$ for some $i \in \{ 1,\ldots,M_1 \}$ and $j \in \{ 1,\ldots,M_2 \}$. Now, since $\mu$ is purely $m$-unrectifiable, Lemma \ref{thm:pertti} implies that there are $x^1,y^1\in E$ such that $y^1 \in X(x^1,V_1,\alpha/3) \setminus H(x^1,\theta_1,\alpha/3)$. This in turn implies that for some $\hhh^1,\jjj^1\in\Sigma_*$, we have $x^1\in E_{\hhh^1}$, $y^1\in E_{\jjj^1}$, and $E_{\jjj^1} \subset X(y,V_1,\alpha/2) \setminus H(y,\theta_1,\alpha/2)$ for all $y\in E_{\hhh^1}$. Now, repeating this argument on $E_{\hhh^1}$, we find $\hhh^2,\jjj^2\in\Sigma_*$ such that $E_{\hhh^1\jjj^2} \subset X(y,V_1,\alpha/2) \setminus H(y,\theta_2,\alpha/2)$ for all $y\in E_{\hhh^1\hhh^2}$. Continuing in this way $M=M_1M_2$ times, we see that $\hhh=\hhh^1\cdots\hhh^M$ fulfills \eqref{coneincl}.

  For each $\iii \in \Sigma_*$, by applying \eqref{coneincl} in $E_{\iii\kkk\hhh}$ for all $\kkk \in \Sigma_{(k-1)l}$, we get the estimate
  \begin{align*}
    \mu\bigl( \bigl\{ x\in E_\iii : \;&\mu\bigl( E_{\iii} \cap X(x,V,\alpha) \setminus H(x,\theta,\alpha) \bigr) \le \gamma^k\mu(E_\iii) \\ &\text{for some } V \in G(n,n-m) \text{ and } \theta \in S^{n-1} \bigr\} \bigr) \le \mu(E_\iii)\bigl( 1 - \mu(E_\hhh) \bigr)^k,
  \end{align*}
  where $\gamma=p^l$, for all $k\in\N$. Let $k_n$ be the integer part of $-2 \log n/\log\bigl( 1-\mu(E_\hhh) \bigr)$ for all $n \in \N$ and define
  \begin{align*}
    A_n = \bigl\{ \pi(\iii) \in E : \;&\mu\bigl(  E_{\iii|_n} \cap X(x,V,\alpha) \setminus H(x,\theta,\alpha) \bigr) \le \gamma^{k_n}\mu(E_{\iii|_n}) \\ &\text{for some } V \in G(n,n-m) \text{ and } \theta \in S^{n-1} \bigr\}.
  \end{align*}
  Since $\sum_{n=1}^\infty \mu(A_n) = \sum_{n=1}^\infty \bigl( 1-\mu(E_\hhh) \bigr)^{k_n} < \infty$ the Borel-Cantelli lemma implies that $\mu$-almost every $x\in E$ belongs to only finitely many $A_n$. This means that for any $s_1 > 2\log\gamma/\log\bigl( 1-\mu(E_\hhh) \bigr)$ we have
  \begin{equation}\label{eq:cylinders}
    \liminf_{n \to \infty} \inf_{\yli{\theta \in S^{n-1}}{V\in G(n,n-m)}} \frac{\mu\bigl( E_{\iii|_n} \cap X(\pi(\iii),V,\alpha) \setminus H(\pi(\iii),\theta,\alpha) \bigr)}{n^{-s_1} \mu\bigl( E_{\iii|_n}\bigr)} = \infty
  \end{equation}
  for $\mu$-almost all $\iii \in \Sigma$.

  Since $\diam(E_{\iii|_n})\le r_{\kappa}^n \diam(E)$ we have $|\log\diam(E_{\iii|_n})| \ge cn$ for some constant $c>0$. Hence $n^{-1}$ in \eqref{eq:cylinders} can be replaced by $|\log\diam(E_{\iii|_n})|^{-1}$. It remains to show that the measure $\mu(E_{\iii|_n})$ in \eqref{eq:cylinders} can be replaced by $\mu\bigl(  B(\pi(\iii),\diam(E_{\iii|_n})) \bigr)$.

  Recalling that $E$ satisfies the open set condition, it follows from \cite[Theorem 2.1]{Schief1994} (see also \cite[Theorem 3.3]{LauRaoYe2001}, \cite[\S 3]{PeresRamsSimonSolomyak2001}, and \cite[Theorem 4.7]{KaenmakiVilppolainen2008}) that there are $\kkk\in\Sigma_*$ and $\delta>0$ such that
  \begin{equation*}
    \dist(E_{\iii\kkk},E \setminus E_\iii) > \delta\diam(E_\iii)
  \end{equation*}
  for all $\iii \in \Sigma_*$. This gives for each $\iii \in \Sigma_*$ and $k\in\N$ an estimate
  \[
    \mu\bigl( \bigl\{ x \in E_{\iii} : B(x,\delta r_{1}^{k|\kkk|}) \cap E \setminus E_\iii \ne \emptyset \bigr\} \bigr) \le \bigl( 1-\mu(E_\kkk) \bigr)^k.
  \]
  Applying Borel-Cantelli similarly as above implies that if $s_2 > |\kkk|\log p/\log\bigl( 1-\mu(E_\kkk) \bigr)$, then
  \[
    \liminf_{n \to \infty} \frac{\mu(E_{\iii|_n})}{n^{-s_2} \mu\bigl( B(\pi(\iii), \diam(E_{\iii|_n})) \bigr)} = \infty
  \]
  for $\mu$-almost every $\iii\in\Sigma$. Combining this with \eqref{eq:cylinders} finishes the proof.
\end{proof}

The following proposition shows that the exponent $s$ indeed depends on $\mu$ and that it is not in general possible to choose $s$ close to $1$.
  \begin{figure}
    \psfrag{f0}{$f_1$}
    \psfrag{f1}{$f_2$}
    \psfrag{f2}{$f_3$}
    \psfrag{f3}{$f_4$}
    \begin{center}
      \includegraphics[width=0.35\textwidth]{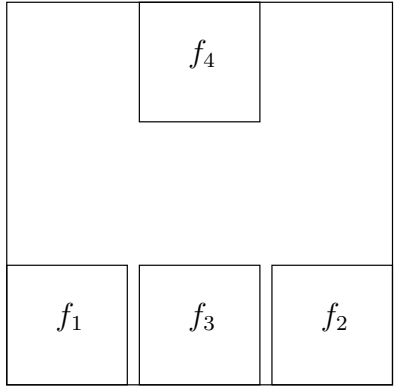}
    \end{center}
    \caption{The first level of the construction of the self-similar set in Proposition \ref{prop:selfsim}.}
    \label{fig:ifs}
  \end{figure}

\begin{proposition}\label{prop:selfsim}
  Let $\tfrac14<\lambda<\tfrac13$ and let $E$ be the self-similar set induced by the similitudes $\{ x \mapsto \lambda x + a_i \}_{i=1}^4$, where $a_1 = (0,0)$, $a_2=(1-\lambda,0)$, $a_3=((1-\lambda)/2,0)$, and $a_4=((1-\lambda)/2,1-\lambda)$; see Figure \ref{fig:ifs}. Suppose $0<p<\tfrac12$ and $\mu$ is the self-similar measure on $E$ corresponding to the Bernoulli weights $p_1=p_2=\tfrac{1-p}{2}$ and $p_3=p_4=\tfrac{p}{2}$. If $\ell$ is the $y$-axis and $\alpha=(1-3\lambda)/10$, then, for a constant $c>0$ independent of $p$, we have
  \[
    \liminf_{r \downarrow 0} \frac{\mu\bigl( B(x,r) \cap X(x,\ell,\alpha) \bigr)}{|\log r|^{-c/p}\mu\bigl( B(x,r) \bigr)} = 0
  \]
  for $\mu$-almost all $x\in E$.
\end{proposition}

\begin{proof}
  For $x = \pi(\iii) \in E$ and $n\in\N$ denote by $Z_n(x)$ the length of the longest subword of $\iii|_n$ that contains only symbols $1$ and $2$. It holds that
  \begin{equation}\label{thm:Revesz}
    \lim_{n \to\infty} \frac{Z_{n}(x)}{\log n}=\frac{1}{|\log(1-p)|}
  \end{equation}
  for $\mu$-almost all $x\in E$. This statement is proved by replacing $\log_2$ in the proof of \cite[Theorem 7.1]{Revesz2005} by $\log_{1/(1-p)}$. 

  Let $\jjj\in\Sigma_n$ and suppose that the subword consisting the last $k$ symbols of $\jjj$ contains only $1$ and $2$. Then it follows that
  \[
    \mu\bigl( E_{\jjj|_{n-k}}\cap X(x,\ell,\alpha) \bigr) \le 2^{-k}\mu(E_{\jjj|_{n-k}})
  \]
  for all $x\in E_\jjj$ and thus, relying on the strong separation condition, we find $c_1>0$ such that also
  \begin{equation}\label{eq:balestimate}
    \mu\bigl( B(x,c_1\lambda^{n-k}) \cap X(x,\ell,\alpha) \bigr) \le 2^{-k}\mu\bigl( B(x,c_1\lambda^{n-k}) \bigr)
  \end{equation}
  for all $x\in E_\jjj$. From \eqref{thm:Revesz} it follows that for $\mu$-almost every $x=\pi(\iii)$, we find infinitely many $n\in\N$ such that for $\jjj=\iii|_n$ the estimate \eqref{eq:balestimate} holds with $k>\tfrac{\log n}{2p}$. A simple calculation then implies the claim for any choice of $0<c<\tfrac12$.
\end{proof}

\begin{remark}
  Given a self-similar measure $\mu$ satisfying the assumptions of Theorem \ref{thm:selfsimilar}, we can define $s(\mu)$ as the infimum of $s>1$ for which the claim of Theorem \ref{thm:selfsimilar} holds. In the view of Proposition \ref{prop:selfsim}, it is natural to ask what is the relation between $s(\mu)$ and the defining similitudes and the Bernoulli weights. This question is interesting already for the four corner Cantor set and its natural measure.
\end{remark}


\end{document}